\documentclass{amsart}

\usepackage{amsmath,amsfonts,amscd,amssymb,enumitem,color}

\def\newmathop#1{\expandafter\gdef\csname #1\endcsname{\mathop{\rm #1}\nolimits}}
\def\newvmathop#1{\expandafter\gdef\csname v#1\endcsname{\mathop{\rm #1}\nolimits}}

\usepackage{longtable,lscape}

\theoremstyle{plain}
\newcounter{thmcount}[section]

\newtheorem{theorem}[thmcount]{Theorem}

\newtheorem{lemma}[thmcount]{Lemma}
\newtheorem{proposition}[thmcount]{Proposition}

\newtheorem*{conjecture*}{Conjecture}
\newtheorem*{theorem*}{Theorem}
\newtheorem*{corollary*}{Corollary}

\newcommand{\thistheoremname}{}
\newtheorem{genericthm}[thmcount]{\thistheoremname}

\newtheorem*{genericthm*}{\thistheoremname}
\newenvironment{namedthm*}[1]
  {\renewcommand{\thistheoremname}{#1}%
   \begin{genericthm*}}
  {\end{genericthm*}}

\theoremstyle{definition}
\newtheorem{remark}[thmcount]{Remark}
\newtheorem*{example}{Example}
\newtheorem*{remark*}{Remark}
\newtheorem{definition}[thmcount]{Definition}

\numberwithin{equation}{section}

\def\Q{{\mathbb Q}}
\def\F{{\mathbb F}}

\def\Qp{{{\mathbb Q}_p}}
\def\Qb{\overline{\mathbb Q}}

\def\Qpb{{\overline{\mathbb Q}_p}}
\def\Z{{\mathbb Z}}

\def\C{{\mathbb C}}

\newmathop{Disc}
\newmathop{Tr}
\newmathop{Norm}
\newmathop{ord}
\newmathop{GL}
\newmathop{SL}
\newmathop{PGL}
\newmathop{Hom}
\newmathop{Ind}
\newmathop{Res}
\newmathop{rk}
\newmathop{corank}
\newmathop{coker}
\newmathop{codim}
\newmathop{cyc}
\newmathop{Reg}
\newmathop{Gal}
\newmathop{Sel}
\newmathop{Frob}
\newmathop{BSD}
\newmathop{tors}

\newmathop{id}
\newmathop{Aut}
\newcommand{\Serreweight}{\sigma}
\newcommand{\BreuilMezard}{\text{BM}}
\newcommand{\DefRing}{R(k,\rho)}
\newcommand{\Localrep}{r}
\newmathop{Art}

\DeclareMathOperator{\JH}{JH}
\renewcommand{\det}{\operatorname{det}}

\newcommand{\mtwo}[4]{\begin{pmatrix} #1 & #2  \\ #3 & #4 \end{pmatrix}}

\usepackage{layout}

\usepackage{graphicx}
\usepackage{cite}
\usepackage{color}
\usepackage{appendix}
\usepackage{amssymb,latexsym,amsthm}
\usepackage{amsfonts}
\usepackage{mathtools}

\usepackage{tikz-cd}
\usetikzlibrary{matrix, calc, arrows}
\usetikzlibrary{arrows.meta}

\newcommand{\Sym}{\text{Sym}}

\DeclareFontFamily{U}{wncy}{}
\DeclareFontShape{U}{wncy}{m}{n}{<->wncyr10}{}
\DeclareSymbolFont{mcy}{U}{wncy}{m}{n}
\DeclareMathSymbol{\Sha}{\mathord}{mcy}{"58}

\input cyracc.def       % sha

\usepackage{xcolor}

\newsavebox\MBox

\begin{document}

\title{Serre weights and the Breuil--M\'{e}zard conjecture for modular forms}

\begin{abstract}
Serre's strong conjecture, now a theorem of Khare and Wintenberger, states that every two-dimensional continuous, odd, irreducible mod $p$ Galois representation $\rho$ arises from a modular form of a specific minimal weight $k(\rho)$, level $N(\rho)$ and character $\epsilon(\rho)$. In this short paper we show that the minimal weight $k(\rho)$ is equal to a notion of minimal weight inspired by the recipe for weights introduced by Buzzard, Diamond and Jarvis in \cite{BDJ}. Moreover, using the Breuil--M\'{e}zard conjecture we show that both weight recipes are equal to the smallest $k \geq 2$ such that $\rho$ has a crystalline lift of Hodge--Tate type $(0,k-1)$.
\end{abstract}

\author{Hanneke Wiersema}
\address{King's College London, Strand, London WC2R 2LS, UK}
\email{hanneke.wiersema@kcl.ac.uk}

\maketitle

\tableofcontents
In 1973, Serre conjectured that all continuous, odd, irreducible mod $p$ representations $\rho$ of the absolute Galois group $G_{\Q}$ of the rationals arise from modular forms. In 1987, Serre published a paper \cite{serreduke} in which he specified the minimal weight $k(\rho)$, minimal level $N(\rho)$ and character $\epsilon(\rho)$ for each $\rho$, such that there should be a form $f$ with these invariants that give rise to $\rho$. The conjecture has since been proven by Khare and Wintenberger, building on work of many others \cite{khare}. 
\par 
There has been substantial interest in generalising Serre's original conjectures. One of the directions this can take is by replacing $G_{\Q}$ by $G_K$, where $K$ is a totally real field. It is a folklore conjecture, stated by Buzzard, Diamond and Jarvis (\cite[Conjecture 1.1]{BDJ}), that representations
\[
\rho: G_{K} \to \GL_2(\overline{\F}_p),
\]
which are continuous, irreducible and totally odd, arise from \emph{Hilbert} modular forms. This setting is beyond the scope of Serre's original conjectures, and in particular there is no explicit description for the weights such forms should have. However, there exists a recipe for a generalised Serre's conjecture for something called \emph{algebraic} modularity. This involves a reformulation of modularity in terms of representation theory, first suggested in work of Ash and Stevens in \cite{ash}. For $K$ a totally real field with $p$ unramified and $\rho$ as above, Buzzard, Diamond and Jarvis define what it means for $\rho$ to be modular in this sense  \cite[Definition 2.1]{BDJ}. The weights we are interested in in this context are irreducible $\overline{\F}_{p}$-representations of $\GL_2(\mathcal{O}_K/p)$ and are called Serre weights.

\par
Moreover, given a representation $\rho$ as above, in the same paper, Buzzard, Diamond and Jarvis give an explicit recipe for all possible weights belonging to $\rho$. The conjecture that is presented in \cite{BDJ} predicts the set of Serre weights $\Serreweight$ for which $\rho$ is modular in this sense, assuming the folklore conjecture. This Serre weight conjecture conjecture has now been proven in \cite{geekisin}, \cite{GLS14} and in \cite{newton}, building on \cite{GLS14}, under some hypotheses. However, for the case relevant to us, $K=\Q$, the conjecture follows from results on Serre's original conjectures. 
\par
Given this compatibility, it should not be a surprise that the recipe in \cite{BDJ} for $K=\Q$ is in some way compatible with the minimal weight $k(\rho)$ as given by Serre. However, a priori, the weights in the recipe in \cite{BDJ} cannot be directly compared to the weight of a modular form. We resolve this by associating an integer to each Serre weight  $\sigma$, and then take the minimum of all of these, denoted by $k_{\min}(W(\rho))$, for each $\sigma$ in the Serre weight set $W(\rho)$ as given in \cite{BDJ}. Using modular representation theory, we are then able to show that this minimum is equal to the minimal weight as suggested by Serre. 
\par
We will be able to show more than this, using the Breuil--M\'{e}zard conjecture. In \cite{BM}, Breuil and M\'{e}zard conjectured that the deformations of representations $\rho$ (in a local setting) of a fixed inertial type $\tau$ with Hodge--Tate
weights $(0, k-1)$ are parameterised by a quotient $R(k,\tau,\rho)$ of the universal deformation ring $R(\rho)$. They also gave a conjectural formula for the Hilbert-Samuel multiplicity of this quotient in terms of representation theoretic information associated to $(k, \tau, \rho)$. This conjecture has been rephrased in many settings and we use the version given and proved by Kisin in \cite{KisinFMC} in many cases.  We note here we do not need the actual value $\mu_{n,m}(\rho)$ appearing in the conjecture, but only need to know when it is non-zero. We will carefully describe Kisin's recipe helping us to determine when $\mu_{n,m}(\rho) \neq 0$.  This allows us to connect modularity of $\rho$ of a given weight to the existence of a crystalline lift with given weights, following an explicit recipe given by Kisin in the same paper. It is known that if $\rho$ is modular of weight $k$ and level $N$ prime to $p$, then $\rho|_{G_{\Q_p}}$ has a crystalline lift with Hodge--Tate weights $(0,k-1)$. This leads us to define a third notion of minimal weight, namely the minimal $k$ for which $\rho$ has a crystalline lift with Hodge--Tate weights $(0,k-1)$, which we will denote as $k_{\text{cris}}(\rho)$.
\par
We emphasise here that all three notions of minimal weight will be completely determined by the restriction $\rho|_{G_{\Q_p}}$, and the main result will be a purely local statement. The main theorem is as follows.
\begin{namedthm*}{Main Theorem}
Let $\rho: G_{\Q} \to \GL_2(\F_p)$ be an continuous, odd and irreducible Galois representation for $p$ an odd prime. Then
\[
k(\rho)=k_{\min}(W(\rho))=k_{\text{cris}}({\rho}).
\]
\end{namedthm*}

\par
The equality $k(\rho)=k_{\min}(W(\rho))$ should follow from the weight part of Serre's conjecture. The main contribution of this paper is to give a direct proof using modular representation theory. The next equality $k_{\text{cris}}(\rho)=k_{\min}(W(\rho))$ is a corollary of the Breuil--M\'{e}zard conjecture for which there is also a purely local proof.
\par
The set-up of this paper is as follows. First we introduce some background and notation, after which we will start introducing each of the minimal weight invariants involved in Section \ref{defsection}. We take care with the details of each recipe, and start with Serre's classical recipe before introducing the one from \cite{BDJ}. After introducing some modular representation theory, we show the compatibility of these two recipes in Section \ref{compsection}. In Section \ref{BMsection}, we will introduce a simplified version of the Breuil--M\'{e}zard conjecture and describe Kisin's recipe. In the final section we show this is compatible with the recipe from Buzzard, Diamond and Jarvis which establishes our main theorem. 

\subsubsection*{Acknowledgements}
I would like to thank Fred Diamond for his support and for suggesting the problem. I would also like to thank James Newton for his guidance and support. Finally I would like to thank Toby Gee, Pol van Hoften and Ashwin Iyengar for many helpful comments.
This work was supported by the Engineering and Physical Sciences Research Council~[EP/L015234/1], through the EPSRC Centre for Doctoral Training in Geometry and Number Theory (the London School of Geometry and Number Theory) at University College London.

\section{Background and notation}

\subsection{General notation} In the following we will study continuous Galois representations $\rho: G_{\Q} \to \GL_2(\overline{\F}_p)$. We will fix an algebraic closure $\Qpb$ and an embedding $\Qb \hookrightarrow \Qpb$, so that we can view $G_{\Qp}$ as a subgroup of $G_{\Q}$. We will identify $G_{\Qp}$ with $G_p$ the decomposition subgroup at $p$ and write $I_p$ for the inertia group at $p$. Moreover, we will write $I_{p,w}$ for the wild inertia subgroup and $I_{p,t}$ for the tame inertia subgroup.
\par
We will often be interested in studying restrictions of our representation, e.g. $\rho|_{G_p}$ or $\rho|_{I_p}$. We write $\omega_n$ for a fundamental character of level $n$, and write $\omega$ for the mod $p$ cyclotomic character, i.e. the unique fundamental character of level $1$.
\par
We will write $\chi_{\text{cyc}}: G_{\Qp} \to \Z_p^{\times}$ for the cyclotomic character. For $\lambda \in \overline{\F}_p^{\times}$, we denote by $\mu_{\lambda}: G_{\Qp} \to \overline{\F}_p^{\times}$ the unramified character sending geometric Frobenius to $\lambda$.
\par
We will briefly recall the $p$-adic Hodge theory notions we will need when discussing the Breuil--M\'{e}zard conjecture.

\subsection{$p$-adic Hodge theory notions}
We let $E/\Q_p$ be a sufficiently large extension with ring of integers $\mathcal{O}$, uniformiser $\varpi$ and residue field $\F$. We let $\Localrep: G_{\Qp} \to \GL_d(E)=\Aut_E(V)$ be a continuous representation with $V$ an $E$-vector space of dimension $d$.

\begin{definition}[Crystalline representations]
Let $V$ be as above. Then we say $V$ is crystalline if $D_{\text{cris}}(V)=(V \otimes_{\Q_p} B_{\text{cris}})^{G_{\Q_p}}$ is free of rank $d$ over $(E \otimes_{{\Q}_p} B_{\text{cris}})^{G_{\Qp}}=E$, where $B_{\text{cris}}$ is one of Fontaine's period rings, see e.g. \cite{bergerintroduction}.
\end{definition}

\begin{definition}[Hodge--Tate weights]
If $V$ is a crystalline representation of $G_{\Q_p}$ over $\Qpb$, then we define the Hodge--Tate weights of $V$ to be the multi-set of integers such that an arbitrary integer $i$ appears with multiplicity
\[
\dim_{\Qpb} (V \otimes_{\Q_p} \C_p(-i))^{G_{\Q_p}},
\]
with $\C_p$ denoting the completion of $\Qpb$ and $\C_p(-i))$ the $-i$-th Tate twist of $\C_p$.
\end{definition}

\begin{example}
With this convention the $p$-adic cyclotomic character $\chi_{\text{cyc}}$ has Hodge--Tate weight one.
\end{example}

\section{Definitions of the weight invariants}
\label{defsection}
In this section we first define the minimal weight invariant that does not need a recipe and is easy to define with the $p$-adic Hodge theory notions in mind. Afterwards we will explicitly define Serre's $k(\rho)$ and study the recipe from \cite{BDJ} to allow us to define $k_{\min}(W(\rho))$ in the next section. Throughout this section we let $p$ be an odd prime and we let ${\Localrep}: G_{\Q_p} \to \GL_2(\overline{\F}_p)$ be any continuous representation. We also let $\rho: G_{\Q} \to \GL_2(\overline{\F}_p)$ be a continuous, odd, irreducible representation and write $\rho_{G_p}$ for its restriction to $G_{\Q_p}$.

\subsection{The invariant $k_{\text{cris}}(\rho)$}

\begin{definition}[Crystalline lift]
We say $\Localrep$ has a crystalline lift of weight $k$ if it has a lift $\tilde{\Localrep}$ to characteristic zero which is crystalline with Hodge--Tate weights $(0,k-1)$.
\end{definition}

\begin{definition}[$k_{\text{cris}}(\rho)$]
\label{kcris}
For a representation $\Localrep$ as above we define  $k_{\text{cris}}({\Localrep})$ to be the least $k \geq 2$ such that ${\Localrep}$ has a crystalline lift of weight $k$.
\par
For a representation $\rho$ as above we set $k_{\text{cris}}(\rho):=k_{\text{cris}}(\rho_{G_p})$.
\end{definition}

\begin{example}
\label{exception}
Suppose ${\rho}_{G_{\Q_p}}$ is given by 
\[
\mtwo{\omega \mu_{\lambda}}{0}{0}{\mu_{\lambda'}},
\]
 then we obtain a lift 
\[
\mtwo{\chi_{\text{cyc}} \mu_{\tilde{\lambda}}}{0}{0} {\mu_{\tilde{\lambda'}}},
\]
with $^\sim$ denoting the Teichm{\"u}ller lift. This has Hodge--Tate weights $(0,1)$, so we have a crystalline lift of weight $2$ and thus $k_{\text{cris}}(\rho)=2$.
\end{example}

\subsection{Serre's minimal weight $k(\rho)$}
\label{oldserreweights}
In this section, we describe how Serre defines $k(\rho)$ in the seminal paper \cite{serreduke}. If $\rho_{G_p}$ is semisimple, $I_{p,w}$ acts trivially by \cite[Proposition 4]{serreprop}. As a consequence the action of $I_p$ factors through $I_{p,t}$. This action $\rho_{I_{p,t}}$ is then given by two characters $\phi,\phi': I_{p,t} \to \overline{\F}_p^{\times}$. From \cite[p. 567]{Edixhoven1992} we obtain that $\{\phi^p, \phi'^p\}=\{ \phi, \phi'\}$. Now, no longer assuming $\rho_{G_p}$ to be semisimple, but looking at the semisimplification of $\rho_{G_p}$ we have two cases:
\begin{enumerate}
\item $\phi, \phi'$ are both of level 1, and $\rho_{G_p}$ is reducible,
\item $\phi, \phi'$ are both of level 2, $\phi^p= \phi'$ and  $\phi'^p= \phi$, and $\rho_{G_p}$ is irreducible.
\end{enumerate}
\subsubsection{The level two case}
In the level two case, we can write 
\begin{align*}
\rho_{I_p}
& \sim \omega^{a} \otimes \begin{pmatrix}
\omega_2^{b-a}  & 0 \\
  0 &  \omega_2^{p(b-a)}
 \end{pmatrix}, 
\end{align*}
with $0 \leq a < b \leq p-2$ (\cite[Section 2.2]{serreduke}). In this case we define
\begin{equation*}
\label{eq:leveltwo}
k(\rho)=pa+b+1.
\end{equation*}
\subsubsection{The level one case}
In the level one case the representation $\rho_{G_p}$ is reducible, and we split into two cases depending on whether $\rho_{I_{p,w}}$ is trivial. 
\par
\vspace{1mm}
  \framebox{$\rho_{I_{p,w}}$ is trivial} \vspace{1mm} \\
In this case, we can write
\begin{align*}
\rho_{I_p}
 \sim \omega^a \otimes \begin{pmatrix}
  \omega^{b-a}  & 0 \\
  0 &  1
 \end{pmatrix},
\end{align*}
for $0 \leq a \leq b \leq p-2$, and in which case we define $k(\rho)=pa+b+1$, however,  if $(a,b)=(0,0)$ this definition would give us $k(\rho)=1$. Since we are not including modular forms of weight $1$, we simply add $p-1$, so $k(\rho)=p$. So we find 
\[
k(\rho)=\begin{cases} pa+b+1 & \text{if } (a,b) \neq (0,0), \\
p & \text{if } (a,b)=(0,0). \end{cases}
\]
\par
\vspace{1mm}
  \framebox{$\rho_{I_{p,w}}$ is non-trivial} \vspace{1mm} \\
Here we can write
\begin{align*}
\rho_{I_p}
 \sim \omega^{\alpha} \otimes \begin{pmatrix}
  \omega^{\beta-\alpha}  & \ast \\
  0 &  1
 \end{pmatrix}. 
\end{align*}
We fix representatives $\alpha, \beta$ such that $0 \leq \alpha \leq p-2$ and $1 \leq \beta \leq p-1$. For $\beta \neq \alpha+1$, we define $a=\min(\alpha,\beta)$ and $b=\max(\alpha,\beta)$, and then we set
\begin{equation*}
\label{eq:levelone}
k(\rho)=pa+b+1.
\end{equation*}
However, if $\beta=\alpha+1$, we split according to whether the representation is \emph{peu ramifi\'{e}e} or \emph{tr\`{e}s ramifi\'{e}e} at $p$. For definitions of these notions we refer to \cite[p. 186]{serreduke}. 
Then Serre defines
\[
k(\rho)=\begin{cases} 2+\alpha(p+1) & \text{ if $\rho_{G_p}$ is peu ramifi\'{e}e and $\beta=\alpha+1$}, \\
p+1+\alpha(p+1) & \text{ if  $\rho_{G_p}$ is tr\`{e}s ramifi\'{e}e and $\beta=\alpha+1$}. \end{cases}
\]
This definition is slightly different for $p=2$, but we focus only on odd primes.

\subsection{Serre weights and the set $W(\rho)$}
Before we introduce the new notion of weights, we briefly reformulate Serre's weight recipe. We write $\Sym^n \overline{\F}_p^2$ for the $n$-th symmetric power of the standard representation of $\GL_2(\F_p)$ on $\overline{\F}_p^2$. Now for any 
continuous, odd, irreducible representation $\rho: G_{\Q} \to \GL_2(\overline{\F}_p)$ we say $\rho$ is modular of weight $k \geq 2$ if and only if  $\rho \sim \overline{\rho}_f$, $f$ of weight $k$ and level $N$ prime to $p$. Now, as is explained in \cite{herziggee}, this is equivalent to the Hecke eigensystem of $\rho$ appearing in $H^1(\Gamma_1(N),\Sym^{k-2}\overline{\F}_p^2)$, which is in turn equivalent to the Hecke eigensystem appearing in $H^1(\Gamma_1(N),V)$ for $V$ a Jordan--H{\"o}lder factor of $\Sym^{k-2}\overline{\F}_p^2$.
\par
If $k > p+1$, then $\Sym^{k-2}\overline{\F}_p^2$ is reducible, so we get proper Jordan--H{\"o}lder factors. We obtain a new notion of weight from \cite{BDJ} as follows.
\begin{definition}[Serre weight]
We define a Serre weight to be an irreducible $\overline{\F}_p$-representation of $\GL_2(\F_p)$. 
\end{definition}
The irreducible representations of $\GL_2(\F_p)$ are given by 
\[
V_{a,b}=\det ^a \otimes \Sym^{b-1}\overline{\F}_p^2,
\]
where $0 \leq a \leq p-2$ and $1 \leq b \leq p$.  We do not need to recall the precise definition of what it means for $\rho$ to be modular of weight $W$, where $W$ is any finite-dimensional $\overline{\F}_p[(\GL_2(\F_p)]$-module, here, since it is not relevant for this paper but refer to  \cite[Definition 2.1]{BDJ}. The authors proceed in showing that this notion of modularity of weight $W$ is equivalent to $\rho$ being modular of weight $V$ where $V$ is a Jordan--H{\"o}lder factor of $W$ \cite[Lemma 2.3]{BDJ}, and these are exactly the Serre weights. For each $\rho$ as above (and more general $\rho$) they define a set $W(\rho)$ consisting of Serre weights and then conjecture that if $\rho$ is modular, then
\[
W(\rho)=\{ V \, \vert \, \rho \text{ is modular of weight } V \}.
\]
Next we will describe the set $W(\rho)$ for all possible $\rho$, up to twists by the cyclotomic character. This is motivated by the compatiblity of modularity in the sense of Serre weights among determinants and twists, proven in \cite[Corollary 2.11(2)]{BDJ}. In particular, $\rho$ is modular of weight $b+1$  and level prime to $p$ in the classical sense if and only if $\omega^a \rho$ is modular of weight $V_{a,b}$ \cite[p. 30]{BDJ}.
\begin{definition}
\label{newweights}
Let  ${\rho: G_{\Qp} \to \GL_2(\overline{\F}_p)}$ be a continuous, odd, irreducible representation, suppose $\rho_{G_p}$ is irreducible, then up to a twist
\[
\rho_{I_p} \sim  \begin{pmatrix}
  \omega_2^{b}   & 0 \\
  0 &   \omega_2^{pb} \end{pmatrix},
\]
for $1 \leq b \leq p-1$. Then the set $W(\rho)$ is given by
\[
W(\rho)=\{ V_{0,b}, V_{b-1,p+1-b}\}.
\]
Otherwise, for $\rho_{G_p}$ reducible, up to a twist we can write
\[
\rho_{I_p} \sim  \begin{pmatrix}
  \omega^{b}   & \ast \\
  0 &  1 \end{pmatrix},
\]
for $1 \leq b \leq p-1$, and we write 
\[
\rho_{G_p} \sim  \begin{pmatrix}
\chi_1   & \ast \\
  0 &  \chi_2 \end{pmatrix}.
\]
In this case, as in the proof of \cite[Theorem 3.17]{BDJ},
\begin{align}
\label{BDJrecipe}
W(\rho)=\begin{cases} 
\{V_{0,b}\}, & \text{if } 1<b < p-1 \text{ and $\rho_{G_p}$ is non-split}, \\
\{V_{0,b},V_{b,p-1-b} \}, &  \text{if } 1<b < p-2 \text{ and $\rho_{G_p}$ is split}, \\
\{V_{0,p-2},V_{p-2,p},V_{p-2,1} \}, &  \text{if } b=p-2, p>3  \text{ and $\rho_{G_p}$ is split}, \\
\{V_{0,p-1}\}, &  \text{if } b= p-1 \text{ and $p>2$}, \\
\{V_{0,p} \}, &  \text{if } b=1, \chi_1 \chi_2^{-1}=\omega \text{ and $\rho_{G_p}$ is tr\`{e}s ramifi\'{e}e}, \\
\{V_{0,p},V_{0,1}, V_{1,p-2} \}, &  \text{if } b=1, p > 3, \text{ and $\rho_{G_p}$ is split}, \\
\{V_{0,3},V_{0,1},V_{1,3},V_{1,1} \}, &  \text{if } b=1, p=3 \text{ and $\rho_{G_p}$ is split}, \\
\{ V_{0,p},V_{0,1} \}, & \text{otherwise.}
\end{cases}
\end{align}
\end{definition}
We emphasise that in fact the above set is depends only on $\rho_{I_p}$, as is proved in \cite[Proposition 3.13]{BDJ}.

\section{Compatibility of two modularity recipes}
\label{compsection}
In this section we will first introduce notions concerning the Grothendieck groups of our representations, which can be found in more in detail in \cite{reduzzi}. Then for each Serre weight $V_{a,b}$ we will define $k_{\min}(V_{a,b})$ as the minimal $k$ such that $V_{a,b}$ appears as Jordan--H{\"o}lder constituent of $\Sym^{k-2} \overline{\F}_p^2$. Using modular representation theory we will show that we can determine this value explicitly. Then we will finally define the third weight invariant of interest, $k_{\min}(W(\rho))$, which will be the minimum of all the $k_{\min}(V_{a,b})$ for $V_{a,b} \in W(\rho)$. 

\subsection{Grothendieck group relations}
We write $G_0(\overline{\F}_p[\GL_2(\F_p)])$ for the Grothen\-dieck group on finite-dimensional representations of $\GL_2(\F_p)$ over $\overline{\F}_p$. The group is isomorphic to the free abelian group generated by the classes of irreducible representations, i.e. $[\det^a \otimes \Sym^n]$ with $0 \leq a \leq p-2$, $0 \leq n \leq p-1$. As in \cite{reduzzi}, write $R \leq R'$ whenever $R'-R$ is in the submonoid of $G_0(\overline{\F}_p[\GL_2(\F_p)])$ consisting of classes of $\overline{\F}_p$-representations of $\GL_2(\F_p)$. If $\sigma, \sigma'$ are $\overline{\F}_p$-representations of $\GL_2(\F_p)$ and if $\sigma$ is irreducible, then $[\sigma] \leq [\sigma']$ if and only if $\sigma$ is a Jordan--H{\"o}lder factor of $\sigma'$. Hence the Grothendieck group enables us to determine the Jordan--H{\"o}lder factors of the representations we are interested in.
\par
For $k < -1$, we define
\[
[\Sym^k]:=-[\det^{k+1} \otimes \Sym^{-k-2}],
\]
and $[\Sym^{-1}]=0$. To ease notation, in the following we will write $S_n$ for $\Sym^n \overline{\F}_p^2$. In all following results we will make extensive use of Serre's periodic relation
\[
[S_{n+p-1}-S_n]=[\det \otimes (S_{n-2} - S_{n-p-1})]
\]
which holds for all $n \in \Z$, obtained from \cite{lettreSerre}.
\begin{lemma}
\label{recursiveformula}
Let $1 \leq n \leq p-1$ and $k \geq 1$, then
\[
[S_{n+k(p-1)}]=[S_n]+[\det^n \otimes S_{p-n-1}]+[\det \otimes S_{n+(k-1)(p-1)-2}].
\]
\end{lemma}
\begin{proof}
We proceed using Serre's periodic relation above and induction on $k$. The base case is $k=1$:
\begin{align*}
    [S_{n+p-1}]&=[S_n]+[\det \otimes S_{n-2}]-[\det \otimes S_{n-p-1}] \\
    &=[S_n]+[\det \otimes S_{n-2}]+[\det \otimes (\det^{n-p} \otimes S_{p-n-1})] \\
    &=[S_n]+[\det \otimes S_{n-2}]+[\det^{n} \otimes S_{p-n-1}].
\end{align*}
Now we suppose that it holds for some $i \geq 1$. Then we find
\begin{align*}
    [S_{n+(i+1)(p-1)}]&=[S_{n+i(p-1)}]+[\det \otimes S_{n+i(p-1)-2}]-[\det \otimes S_{n+i(p-1)-p-1}] \\
    & \overset{\text{(IH)}}{=} [S_n]+[\det^n \otimes S_{p-n-1}]+[\det \otimes S_{n+(i-1)(p-1)-2}]\\
    &\quad +[\det \otimes S_{n+i(p-1)-2}]-[\det \otimes S_{n+i(p-1)-p-1}] \\
    &\,=[S_n]+[\det^n \otimes S_{p-n-1}]+[\det \otimes S_{n+i(p-1)-2}]
\end{align*}
where the last step simply is because $[\det \otimes S_{n+(i-1)(p-1)-2}]=[\det \otimes S_{n+i(p-1)-p-1}]$ since the indices are the same.
\end{proof}
When it is clear from the context, in the following we will often omit the brackets of a class in the Grothendieck group.
\subsection{ The invariants $k_{\min}(V_{a,b})$ and $k_{\min}(W(\rho))$}
Next for each $V_{a,b} \in W(\rho)$, we define
\[
k_{\min}(V_{a,b})=\min_{k} \{ k \geq 2 \, \vert \, V_{a,b} \in \JH(\Sym^{k-2} \overline{\F}_p^2) \},
\]
and using the results above we prove an explicit formula for this.

\begin{proposition}
\label{proofkmin}
Let $0 \leq a \leq p-2$ and $1 \leq b \leq p$, then
\[
k_{\min}(V_{a,b})= \begin{cases} a(p+1)+b+1, & a+b < p, \\
(a+1)(p+1)+bp-p^2, & a+b \geq p. \end{cases}
\]
\end{proposition}
\begin{proof}
Define $k_{\min}'(V_{a,b})$ to be
\[
k_{\min}'(V_{a,b})= \begin{cases} a(p+1)+b+1, & a+b < p, \\
(a+1)(p+1)+bp-p^2, & a+b \geq p. \end{cases}
\]
So we need to show that $k_{\min}(V_{a,b})=k_{\min}'(V_{a,b})$. To avoid complications later, we first prove the proposition for $a=0$.  
\par
For $a=0$ we find $k_{\min}'(V_{a,b})=b+1$ for all $b$ in which case $V_{a,b}$ equals $S_{b-1}$, which is irreducible. Since $S_{b+1-2}=S_{b-1}$, the minimal $k$ such that $V_{a,b}$ appears in some $S_{k-2}$ is $b+1$, so we find $k_{\min}(V_{a,b})=b+1$ and the result follows. Henceforth we assume $a \neq 0$.
\par
Now first we prove $V_{a,b} \in \JH(S_{k_{\min}'(V_{a,b})-2})$, starting with $a,b$ such that $a+b < p$. We do this by induction on $a$, the base case is $a=0$, which is done above. Now suppose we have proved for some $a \leq p-2$ that if $1 \le b < p-a$, that then $V_{a,b} \in \JH(S_{k_{\min}'(V_{a,b})-2})$. Now we use Lemma \ref{recursiveformula} to show that if $a+1 \le p-2$ and $1 \le b < p-(a+1)$, that then $V_{a+1,b} \in \JH(S_{k_{\min}'(V_{a+1,b})-2})$:
\begin{align*}
S_{k_{\min}'(V_{a+1,b})-2} &=S_{(a+1)(p+1)+b-1} \\
& \geq \det \otimes S_{a(p+1)+b-1} \\
& \overset{\text{(IH)}}{\geq} \det \otimes (\det^a \otimes S_{b-1}) \\
&= \det^{a+1} \otimes S_{b-1}.
\end{align*}
Now for $a,b$ such that $a+b \geq p$, we again perform induction on $a$. This time our base case is $a=p-b$ in which case $k_{\min}'(V_{a,b})=a+p+1$. Now we use Lemma \ref{recursiveformula} again, which gives us
\begin{equation}
\label{eq:casep}
S_{a+p-1}=S_a+\det \otimes S_{a-2}+\det^a \otimes S_{p-a-1}.
\end{equation}
Now since $b=p-a$, we find that $V_{a,b}$ corresponds to the last term. Now again we do induction on $a$, suppose we have proved for some $a \le p-2$ that if $p-a \leq b \leq p$, that then $V_{a,b} \in \JH(S_{k_{\min}'(V_{a,b})-2})$. Again we use Lemma \ref{recursiveformula}, this time to show that if $a+1 \le p-2$ and $p-(a+1) \leq b \leq p$, that then $V_{a+1,b} \in \JH(S_{k_{\min}'(V_{a+1,b})-2})$:
\begin{align*}
S_{k_{\min}'(V_{a+1,b})-2} &=S_{(a+2)(p+1)+bp-p^2-2} \\
& \geq \det \otimes S_{(a+1)(p+1)+bp-p^2-2} \\
& \overset{\text{(IH)}}{\geq} \det \otimes (\det^a \otimes S_{b-1}) \\
&= \det^{a+1} \otimes S_{b-1}.
\end{align*}
This finishes the first part, we proceed with showing that the values given above for $k_{\min}'(V_{a,b})$ are actually minimal.   
\par
We will show that if $r < k_{\min}'(V_{a,b})$, then $V_{a,b} \not \in \JH(S_{r-2})$. Now we write ${r=s+t(p+1)}$ for some $2 \leq s \leq p+2$ and $t \geq 0$. We proceed by using induction on $t$. First say $t=0$, then for $r=s$ and for $2 \leq r \leq p+1$, $S_{r-2}$ is irreducible, in which case $V_{a,b} \in \JH(S_{r-2})$ if and only if $V_{a,b}=S_{r-2}$. But the latter implies $a=0$, which we have already dealt with above. Now for the remaining case $r=p+2$, we find
\[
S_{r-2}=S_p=S_1 + \det \otimes S_{p-2},
\]
hence we have two Jordan--H{\"o}lder constituents. Now since $S_1=V_{0,2}$, again we have $a=0$. Now for the second constituent we find $k_{\min}'(V_{1,p-1})=p+2$. But then $r=k_{\min}'(V_{a,b})$, which gives us a contradiction. 
\par
Now suppose it holds for some $r=s+t(p+1)$, then to complete our induction we need to prove that if $r=s+(t+1)(p+1) < k_{\min}'(V_{a,b})$, then $V_{a,b} \not \in \JH(S_{r-2})$. Again we use Lemma \ref{recursiveformula}, first we rewrite $r=A(p-1)+B+p+1$ for some $A \geq 0$ and $1 \leq B \leq p-1$. Then we obtain
\begin{align*}
S_{r-2}&=S_{A(p-1)+B+p+1-2}=S_{(A+1)(p-1)+B} \\
&=S_B+ \det^B \otimes S_{p-B-1} +\det \otimes S_{A(p-1)+B-2}.
\end{align*}
The first two terms in the sum are irreducible, so are Jordan--H{\"o}lder constituents, hence we have to show these do not correspond to $V_{a,b}$. Recall that we have assumed $a \neq 0$, hence $V_{a,b} \neq S_B$ for any $B$. Next we consider the term $\det^B \otimes S_{p-B-1}$. Now since $1 \leq B \leq p-1$ we must have $B=a$ (as $a \neq 0$) and hence $b=p-B=p-a$ in order for it to correspond to $V_{a,b}$. But for $V_{a,p-a}$ we have $k_{\min}'(V_{a,b})=a+p+1$. However, since $t+1 \geq 1$, this implies $r \geq a+p+1$, which gives us a contradiction with $r < k_{\min}'(V_{a,b})$. 
\par
It remains to show $V_{a,b} \not \in \JH(\det \otimes S_{A(p-1)+B-2})$, which is equivalent to showing $V_{a-1,b} \not \in \JH(S_{A(p-1)+B-2})$. Now by our induction hypothesis we find that if ${r-(p+1)<k_{\min}'(V_{a-1,b})}$ then $V_{a-1,b} \not \in \JH(S_{r-(p+1)-2})$. Now for all $a+b \neq p$, we find $k_{\min}'(V_{a-1,b})=k_{\min}'(V_{a,b})-(p+1)$, so the result follows for such $a,b$ where $a+b \neq p$. 
\par
Now we finish by showing $V_{a,b} \not \in \JH(\det \otimes S_{A(p-1)+B-2})$ in the remaining case where $a+b=p$. We have seen that $k_{\min}'(V_{a,p-a})=a+p+1$. We further know that for $a \neq 0$, we have $k_{\min}(V_{a,b}) \geq p+2$. Now if  for some $r < k_{\min}'(V_{a,b})$, we have $V_{a,b} \in \JH(S_{r-2})$, then we must have $p+2 \leq r \leq a+p$. Recalling $r=A(p-1)+B+p+1$   for some $A \geq 0$ and $1 \leq B \leq p-1$, this gives us $A=0$ and $1 \leq B \leq a-1$, so that $\det \otimes S_{A(p-1)+B-2}= \det \otimes S_{B-2}$, which is irreducible and does not correspond to $V_{a,p-a}$. So it follows in this case too and we can conclude $k_{\min}(V_{a,b})=k_{\min}'(V_{a,b})$.
\end{proof}

Now we finally introduce the third minimal weight invariant:
\begin{definition}
\label{kmin}
We set
\[
k_{\min}(W(\rho))=\min_{k} \{ k_{\min}(V_{a,b}) \, \vert \, V_{a,b} \in W(\rho) \}.
\]
Alternatively, we can write
\[
k_{\min}(W(\rho))=\min_k \{ k \geq 2 \, \vert \, W(\rho) \cap \JH(\Sym^{k-2} \overline{\F}_p^2) \neq \emptyset \}.
\] 
\end{definition}
Now we are ready to prove the compatiblity between the two weight recipes.

\begin{theorem}
\label{finaltheorem}
Let $p$ be an odd prime and let 
\[
\rho: G_{\Q} \to \GL_2(\overline{\F}_{p}),
\]
 be a continuous, odd, irreducible representation.  Let $k(\rho)$ be defined as in Section \ref{oldserreweights} and $k_{\min}(W(\rho))$ as above, then
\[
k(\rho)=k_{\min}(W(\rho)).
\]
\end{theorem}
\begin{proof}
We go through it case by case. We start with the irreducible case where
\[
\rho_{I_p} \sim  \omega^a \otimes \begin{pmatrix}
  \omega_2^{(b-a)}   & 0 \\
  0 &   \omega_2^{p(b-a)} \end{pmatrix},
\]
for some $0 \leq a < b \leq p-1$. In this case $k(\rho)=pa+b+1$ and we have ${W(\rho)=\{V_{a,b-a},V_{b-1,p+1-(b-a)} \}}$. Now by Proposition \ref{proofkmin} we find
\[
k_{\min}(V_{a,b-a})=a(p+1)+b-a+1=pa+b+1, \quad k_{\min}(V_{b-1,p+1-(b-a)})=pa+b+p,
\]
hence
\[
k_{\min}(W(\rho))=pa+b+1=k(\rho).
\]
Next we continue with the reducible case, first assume $\rho_{I_{p,w}}$ is trivial, then
\[
\rho_{I_p} \sim   \omega^a \otimes \begin{pmatrix}
  \omega^{b-a}   & 0 \\
  0 &   1 \end{pmatrix},
\]
for $0 \leq a  \leq b \leq p-2$. Here we have $k(\rho)=pa+b+1$ unless $(a,b)=(0,0)$ in which case $k(\rho)=p$. We have 
\[
W(\rho)=\begin{cases} \{V_{a,b-a},V_{b,p-1-(b-a)}\}, & 1<b-a<p-2, \\
					\{V_{0,p-2},V_{p-2,p},V_{p-2,1} \}, & b-a=p-2, p > 3, \\
					\{ V_{a,p-1} \}, & b-a=0, \\
					\{ V_{a,p}, V_{a,1},V_{a+1,p-2} \}, & b-a=1, p>3, \\
					\{ V_{0,3},V_{0,1},V_{1,3},V_{1,1} \}, & b-a=1, p=3,
\end{cases}
\]
and use Proposition \ref{proofkmin} to find
\[
k_{\min}(W(\rho))=\begin{cases} pa+b+1, & 1<b-a<p-2, \\
								p-1, & b-a=p-2, p > 3, \\
					p, & b-a=0, (a,b)=(0,0), \\
pa+b+1, & b-a=0, (a,b) \neq (0,0), \\
					pa+b+1, & b-a=1, p>3, \\
					2, & b-a=1, p=3.
\end{cases}
\]
In the first, third, fourth and fifth case it is clear that this coincides with $k(\rho)$. For the second case, note $0 \leq a \leq b \leq p-2$ by assumption, hence $a=0, b=p-2$, in which case hence $k(\rho)=p-2+1=p-1$. Finally, if $p=3$, $b-a=1$, then $a=0,b=1$, so $k(\rho)=2$. \\
\par
Next assume $\rho_{I_{p,w}}$ is non-trivial, then
\[
\rho_{{I_p}}\sim \omega^{\alpha} \otimes \begin{pmatrix}
  \omega^{\beta-\alpha}   & \ast \\
  0 &  1
 \end{pmatrix},
\]
where $0 \leq \alpha \leq p-2$ and $1 \leq \beta \leq p-1$. First assume $\beta \neq \alpha+1$, then
\[
k(\rho)=1+p \min(\alpha, \beta)+ \max(\alpha,\beta), 
\]
however, we first assume $\beta \geq \alpha$. We find
\[
W(\rho)=\begin{cases} \{V_{\alpha,\beta-\alpha} \}, & 1<\beta-\alpha<p-1, \\
					\{ V_{\alpha,p-1} \}, & \beta-\alpha=0, \\
					\{ V_{0,p-1} \}, & \beta-\alpha=p-1, \\
\end{cases}
\]
and use Proposition \ref{proofkmin} to find
\[
k_{\min}(W(\rho))=\begin{cases} p\alpha+\beta+1,  & 1<\beta-\alpha<p-1, \\
												p\alpha+\beta+1, & \beta-\alpha=0, \\
					p, & \beta-\alpha=p-1. \\

\end{cases}
\]
Note in the final case $\beta=p-1$, $\alpha=0$, so $k(\rho)=1+p-1=p$, hence in all cases $k_{\min}(W(\rho))=k(\rho)$. 
\par
Now suppose $\beta-\alpha < 0$, then $1 < \beta-\alpha+p-1 \leq p-2$, and so we have $W(\rho)=\{ V_{\alpha, \beta-\alpha+p-1}\}$, now since $\alpha+\beta-\alpha+p-1=\beta+p-1 \geq p$, we find by Proposition \ref{proofkmin}
\[
k_{\min}(V_{\alpha, \beta-\alpha+p-1})=1+p \beta+\alpha
\]
which coincides with $k(\rho)$ in this case. 
Now suppose $\beta=\alpha+1$, then recall that
\[
k(\rho)=\begin{cases} 
2+\alpha(p+1), & \text{ if $\rho_{G_p}$ is peu ramifi\'{e}e}, \\
p+1+\alpha(p+1), & \text{ if $\rho_{G_p}$ is tr\`{e}s ramifi\'{e}e}, \end{cases}
\]
and 
\[
W(\rho)=\begin{cases} \{V_{\alpha,1}, V_{\alpha,p} \}, & \text{if $\rho_{G_p}$ is peu ramifi\'{e}e}, \\
					\{ V_{\alpha,p} \}, & \text{if $\rho_{G_p}$ is tr\`{e}s ramifi\'{e}e}, \\
\end{cases}
\]
and
\[
k_{\min}(W(\rho))=\begin{cases} 2+\alpha(p+1), & \text{if $\rho_{G_p}$ is peu ramifi\'{e}e}, \\
					p+1+\alpha(p+1), & \text{if $\rho_{G_p}$ is tr\`{e}s ramifi\'{e}e}. \\
\end{cases}
\]
by Proposition \ref{proofkmin}, hence the two notions coincide.
\end{proof}

\begin{remark}
Note that in general the minimal Serre weights are not unique in the sense that there may exist two different weights $V_{a,b}, V_{a',b'}$ such that $k_{\min}(V_{a,b})=k_{\min}(V_{a',b'})=k_{\min}(W(\rho))$. For example, let $\rho$ be so that $\rho_{I_{p,w}}$ is trivial, $1 < b-a <p-2$ and $a \neq 0$. In this case $k_{\min}(W(\rho))=1+pa+b$ and the weights $V_{a,b-a}$ and $V_{b,p-1-(b-a)}$ are both minimal Serre weights. 
\end{remark}

\section{Breuil--M\'{e}zard conjecture}
\label{BMsection}
In the previous sections we have purely focused on modularity, but now will use the Breuil--M\'{e}zard conjecture to relate Serre's conjectures to the existence of crystalline lifts. This involves yet another recipe, described by Kisin in \cite{KisinFMC}. We will adjust notation to our setting.
\par
As before, we let $\rho: G_{\Q} \to \GL_2(\overline{\F}_p)$ be a continuous, odd, irreducible representation and $\rho_{G_p}$ its restriction to $G_{\Q_p}$. To each $\rho_{G_p}$, or indeed any continuous representation $\Localrep: G_{\Q_p} \to \GL_2(\overline{\F}_p)$, Kisin associates a non-negative integer $\mu_{n,m}(\rho_{G_p})$ and we will set $\mu_{n,m}(\rho):=\mu_{n,m}(\rho_{G_p})$. 
\par
We will show that having $\mu_{n,m} (\rho)>0$ corresponds to the existence of a crystalline lift of $\rho_{G_p}$ of some specific weight. In order to do that we use the connection between the Breuil--M\'{e}zard conjecture and the set of Serre weights $W(\rho)$ from \cite{BDJ}, which was first observed in \cite{geekisin}.
We will define a set $\BreuilMezard(\rho)$ consisting of all $V_{m,n+1}$ such that $\mu_{n,m}(\rho)>0$. Then, inspired by \cite[Theorem A(2)]{geekisin}, we will show that Kisin's recipe is compatible with \ref{BDJrecipe}, i.e. that $\BreuilMezard(\rho)=W(\rho)$.
\par
The numbers $\mu_{n,m}(\rho)$ are related to the representations we have previously seen introduced as Serre weights, recall these are just the irreducible $\GL_2(\F_p)$ representations over $\overline{\F}_p$. Kisin's notation is as follows:
 \[
 \sigma_{n,m}=\Sym^n \F^2 \otimes \det^m, \quad n \in \{0,\dots,p-1\},m \in \{ 0, \dots, p-2 \}.
\]
Recall that for \cite{BDJ} we get $V_{a,b}=\det^a \otimes \Sym^{b-1} \F^2, \quad a \in \{0,\dots,p-2\},b \in \{ 1, \dots, p \}$, so that we have
\[
\sigma_{n,m}=V_{m,n+1}.
\]
For any continuous representation $\Localrep: G_{\Q_p} \to \GL_2(\overline{\F}_p)$ the numbers $\mu_{n,m}(\Localrep)$ defined by Kisin provide information about the deformation ring $R_{\text{cr}}^{\square,\psi}(k, \tau,\Localrep)$ (\cite[p. 646]{KisinFMC}).
However, since we are only interested in crystalline lifts of $\rho_{G_p}$, this means that our type $\tau$ as defined in \cite[(1.1)]{KisinFMC} is trivial. This also fixes the determinant condition, hence we will drop the character $\psi$ from our notation in the ring and just use $\DefRing$.
\begin{definition}
Let $\DefRing$ be the ring parameterising crystalline lifts of $\rho_{G_p}$ with Hodge--Tate weights $(0,k-1)$.
\end{definition}
Now we state the Breuil--M\'{e}zard conjecture formally as a theorem and give references for the proof.
\begin{theorem}[Breuil--M\'{e}zard conjecture]
\label{BMconj}
Let $p$ be an odd prime. Suppose $\rho$ is as above then
\[
e(\DefRing/ \varpi)= \sum_{n,m} a_{\text{cr}}(n,m) \mu_{n,m}(\rho).
\]
\end{theorem}
This was proven for all $p$ when $\rho_{G_p}$ is not a twist of an extension of the trivial representation by $\omega$ and in all cases when $p \geq 5$, in \cite{KisinFMC},\cite{Paskunas} and \cite{HuTan} with the remaining cases proven in \cite{Tung2018} and \cite{Tung2019}.
We explain the terms on the right hand side first. Applying the definition on \cite[p. 646]{KisinFMC} to our setting, it follows that the $a_{\text{cr}}(n,m)$ in above formula are simply the Jordan--H{\"o}lder multiplicities of the $\sigma_{n,m}$ in $\Sym^{k-2} \overline{\F}_p^2$. We will define the $\mu_{n,m}(\rho)=\mu_{n,m}(\rho_{G_p})$ in the next section. Finally the expression  $e(\DefRing/ \varpi)$ denotes the Hilbert-Samuel multiplicity of $\DefRing/ \varpi$. 

\par 

The only property of the Hilbert-Samuel multiplicity occurring in the conjecture we will need is that $e(\DefRing/ \varpi) \neq 0$ if and only if $\DefRing/ \varpi$ is non-trivial. As a consequence, for our purposes we are not interested in the exact value of $\mu_{n,m}(\rho)$, we only want to know when it is non-zero.
\begin{definition}
We define
\[
\BreuilMezard({\rho})= \{  V_{m,n+1} \, \vert \, \mu_{n,m}({\rho}) > 0 \}.
\]
\end{definition}

Next we want to show that $\BreuilMezard({\rho})=W({\rho})$.  This will allow us to determine the $k$ such that $\DefRing \neq 0$, so that we indeed find a crystalline lift of type $(0, k_{\min}(W(\rho))-1)$. We will use this to show $k_{\text{cris}}(\rho) \geq k_{\min}(W(\rho))$.

\subsection{Kisin's recipe}
As with the other two recipes, we split into the irreducible and reducible case. 
First assume ${\rho}_{G_p}$ is irreducible, where Kisin sets $\mu_{n,m}({\rho})=1$ if
\[
\rho_{I_p} \sim  \omega^m \otimes \begin{pmatrix}
  \omega_2^{(n+1)}   & 0 \\
  0 &   \omega_2^{p(n+1)} \end{pmatrix},
\]
 and $\mu_{n,m}({\rho})=0$ otherwise. 
 \par
Next we continue with the reducible case, Kisin sets   $\mu_{n,m}({\rho})=0$ unless
\[
\rho_{G_p}  \sim   \omega^m \otimes \begin{pmatrix}
  \omega^{n+1} \mu_{\lambda}   & \ast \\
  0 &   \mu_{\lambda'} \end{pmatrix},
\]
for $\lambda, \lambda' \in \overline{\F}_p^{\times}$ in which case we set $\mu_{n,m}({\rho})=1$ except in the following cases:
\begin{enumerate}
\item If $n=p-1$, $\lambda=\lambda'$ and $\ast$ is peu ramifi\'{e}e (including the trivial case), then $\mu_{n,m}(\rho)=2$.
\item If $n=0$, $\lambda=\lambda'$ and $\ast$ is tr\`{e}s ramifi\'{e}e, then $\mu_{n,m}({\rho})=0$.
\item If $n=p-2$, and $\rho_{G_p}$ is semi-simple, then $\mu_{p-2,m}({\rho})=2$ if $\lambda \neq \lambda'$.
\end{enumerate}
 If $n=p-2$, and $\rho_{G_p}$ is semi-simple and $\lambda=\lambda'$, Kisin does not define $\mu_{p-2,m}({\rho})$ explicitly, but defines it to be the Hilbert-Samuel multiplicity obtained by taking $\tau$ trivial and $k=p$ in his conjecture (\cite[Conjecture 1.1.5]{KisinFMC}). Kisin notes that global considerations suggest the integer $\mu_{p-2,m}({\rho})$ to be 2 if $\lambda=\lambda'$.  However, in \cite[Theorem 1]{Sander}, this multiplicity has been computed explicitly. We will adjust Kisin's recipe to account for that, in line with the errata of \cite{KisinFMC} in \cite[Appendix B]{geekisin}, so:
\begin{enumerate}[resume]
\item If $n=p-2$, and $\rho_{G_p}$ is semi-simple and if $\lambda=\lambda'$ we set $\mu_{p-2,m}({\rho})=4$.
\end{enumerate}

\subsection{Compatibility with modularity}
Next we will show $\BreuilMezard({\rho})=W({\rho})$. For simplicity we will not distinguish between $V_{m+p-1,n}$ and $V_{m,n}$ in the proofs below. We start with the irreducible case.

\begin{proposition}
\label{irredequal}
Suppose ${\rho}_{G_p}$ is irreducible, then  $\BreuilMezard({\rho})=W({\rho})$.
\end{proposition}
\begin{proof}
We obtain
\[
\BreuilMezard({\rho})= \{ V_{m,n+1}, V_{m+n,p-n} \}. 
\]
Now for
\[
\rho_{I_p} \sim   \begin{pmatrix}
  \omega_2^{(n+1)}   & 0 \\
  0 &   \omega_2^{p(n+1)} \end{pmatrix},
\]
we have $W({\rho})=\{ V_{0,n+1}, V_{n,p-n} \} $ so that if
\[
\rho_{I_p} \sim  \omega^m \otimes \begin{pmatrix}
  \omega_2^{(n+1)}   & 0 \\
  0 &   \omega_2^{p(n+1)} \end{pmatrix},
\]
we have $W({\rho})=\{ V_{m,n+1}, V_{m+n,p-n} \}$. This completes the proof for the irreducible case.
\end{proof}

\begin{proposition}
\label{redequal}
Suppose ${\rho}_{G_p}$ is reducible, then  $\BreuilMezard({\rho})=W({\rho})$.
\end{proposition}
\begin{proof}
In the reducible case it is a bit more complicated to determine $\BreuilMezard({\rho})$. Most of the time we just get one non-zero element, $V_{m,n+1}$, in $\BreuilMezard({\rho})$. But we note that if $\rho$ is split, then we have
\[
\omega^{m+n+1} \mu_{\lambda} \oplus \omega^m \mu_{\lambda'}=\omega^{m+n+1}(\omega^{p-1-(n+1)} \mu_{\lambda'} \oplus \mu_{\lambda}).
\]
Given this, computing $\BreuilMezard({\rho})$ is just a combinatorial exercise, recalling $\omega^{p-1}=1$ and thus $\omega^{p}=\omega$ (which is relevant since we allow $n$ in a range including both $0$ and $p-1$). We obtain

\[
\BreuilMezard({\rho})= \begin{cases} 
V_{m,n+1}, & 0 < n<p-2, \rho_{G_p} \text{ non-split}, \\
V_{m,n+1}, V_{m+n+1,p-2-n}, & 0< n < p-3, \rho_{G_p} \text{ split},\\
V_{m,p-2}, V_{m+p-2,p}, V_{m+p-2,1}, & n=p-3, \rho_{G_p} \text{ split}, p>3, \\
V_{m,n+1}, & n=p-2, \\
V_{m,1},V_{m,p},V_{m+1,1}, V_{m+1,p-1}, & n=0,p-1, \rho_{G_p} \text{ split}, p=3, \\
V_{m,1},V_{m,p},V_{m+1,p-2}, & n=0,p-1, \rho_{G_p} \text{ split}, p>3, \\
V_{m,1},V_{m,p}, & n=0,p-1, \rho_{G_p}  \text{ non-split and not tr\`{e}s ramifi\'{e}e,} \\
V_{m,p}, & n=0,p-1, \rho_{G_p}  \text{ tr\`{e}s ramifi\'{e}e.}  
\end{cases}
\]
The result now follows easily from comparing the above to \eqref{BDJrecipe} and recalling that $ \omega^m \otimes V \in W(\rho)$ is equivalent to $V \in W(\omega^m \otimes \rho)$ due to the compatibility with twists, proven in \cite[Corollary 2.11(2)]{BDJ}.
\end{proof}

\section{The main result}
\label{finalsection}
In this final section we will tie it all together. The main result left to prove is the equality $k_{\min}(W(\rho))=k_{\text{cris}}(\rho)$, which we will do using Proposition \ref{irredequal} and Proposition \ref{redequal} and Theorem \ref{BMconj}. We note here that this is a direct consequence of Conjecture 4.2.2 in \cite{Gee} for $F=\Q$. We refer also to the discussion in the final paragraph of \cite[Section 4.2]{Gee}, however we do not need the assumptions therein with the recent progress on the Breuil--M\'{e}zard  conjecture as made evident in Theorem \ref{BMconj}.

\begin{theorem}
\label{crisequal}
Let $p$ be an odd prime and let
\[
\rho: {G_{\Q}} \to \GL_2(\overline{\F}_p)
\]
be a continuous, odd, irreducible representation. Let $k_{\min}(W(\rho))$ be defined as in Definition \ref{kmin} and  $k_{\text{cris}}(\rho)$ as in Definition \ref{kcris}. Then
\[
k_{\min}(W(\rho))=k_{\text{cris}}(\rho).
\]
\end{theorem}
\begin{proof}
Suppose first we have a crystalline lift of weight $k_{\text{cris}}(\rho)$, then we know $\DefRing$ is non-trivial, hence $e(\DefRing/ \varpi) \neq 0$ for $k=k_{\text{cris}}(\rho)$. By Theorem \ref{BMconj}, this means there must be $(m,n)$ such that $a_{\text{cr}}(n,m)>0$ and $\mu_{n,m}(\rho)>0$, so $V_{m,n+1} \in \BreuilMezard(\rho)=W(\rho)$. Since $a_{\text{cr}}(n,m)>0$, this $V_{m,n+1}$ must be a Jordan--H{\"o}lder constituent of $\Sym^{k_{\text{cris}}(\rho)-2} \overline{\F}_p^2$, so that $k_{\min}(W(\rho)) \leq k_{\text{cris}}(\rho)$. 
\par
Conversely suppose $V_{m,n+1} \in W(\rho)$ is such that $k_{\min}(W(\rho))=k_{\min}(V_{m,n+1})$. Then $V_{m,n+1} \in \BreuilMezard(\rho)$, so that $\mu_{n,m}(\rho)>0$ and that $a_{\text{cr}}(n,m)>0$, so again by Theorem \ref{BMconj} we have $e(\DefRing/ \varpi) \neq 0$. This means  $\DefRing$ is non-trivial for $k=k_{\min}(V_{m,n+1})$, hence we have a crystalline lift of weight $k_{\min}(W(\rho))$. This means $k_{\min}(W(\rho)) \geq k_{\text{cris}}(\rho)$. 
\end{proof}

Now the main result follows from Theorem \ref{crisequal} and Theorem \ref{finaltheorem}.

\begin{theorem}
\label{mainthm}
Let $\rho: G_{\Q} \to \GL_2(\F_p)$ be an continuous, odd and irreducible Galois representation for $p$ an odd prime. Then
\[
k(\rho)=k_{\min}(W(\rho))=k_{\text{cris}}({\rho}).
\]
\end{theorem}

\bibliography{references}
\bibliographystyle{plain}
\end{document}